\documentclass[11pt]{amsart}
  \usepackage[ngerman,english]{babel}
  \usepackage[latin1]{inputenc}
  \usepackage{amsmath}
  \usepackage{amsthm}
  \usepackage{amssymb}
  \usepackage{accents}
  \usepackage{enumerate,enumitem}
  \usepackage{graphicx}
  \usepackage{wrapfig}
  \usepackage{tikz-cd}
  \tikzcdset{cells={font=\everymath\expandafter{\the\everymath\displaystyle}},}
  \usepackage{tikz-3dplot}
  \usetikzlibrary{decorations.pathreplacing,backgrounds}
  \usetikzlibrary{automata, patterns, calc}
\RequirePackage[colorlinks,citecolor=blue,urlcolor=blue]{hyperref}
  \usepackage[arrow, matrix, curve]{xy}
  \usepackage{afterpage}
  \usepackage{subfig}
  \usepackage{setspace}
  \newtheorem{thm}{Theorem}[section]
  \newtheorem{lemma}[thm]{Lemma}
  \newtheorem{prop}[thm]{Proposition}
  \newtheorem{rem}[thm]{Remark}
  \newtheorem{cor}[thm]{Corollary}
  
  \newtheorem{ex}[thm]{Example}
  \newtheorem{mydef}[thm]{Definition}
  \newtheorem*{mydef*}{Definition}
  
  \newcommand{\cone}[1]{\text{C}(#1)}
  \DeclareMathOperator{\id}{id}

  \newcommand{\pp}{\bar{p}}
  \newcommand{\qq}{\bar{q}}

  \newcommand{\IS}{\mathbf{IS}_{\pp}^\bullet}
  \newcommand{\ISq}{\mathbf{IS}_{\qq}^\bullet}
  \newcommand{\VD}{\mathcal{D}}

  \newcommand{\Q}{\mathbb{Q}}

  \newcommand{\C}{\mathbb{C}}

  \newcommand{\Dbc}{\mathcal{D}^b_{cc}} 
  \newcommand{\rZ}{\ring{Z}}

  \newcommand{\OI}{\Omega I_{\pp}}

  \newcommand{\KK}{\mathbf{K}^\bullet}
  \newcommand{\LL}{\mathbf{L}^\bullet}
  \newcommand{\GG}{\mathbf{G}^\bullet}
  \newcommand{\II}{\mathbf{I}^\bullet}
  \newcommand{\BB}{\mathbf{B}^\bullet}

  \newcommand{\TT}{\mathbb{T}}
  
\begin{document}

  \title{Self-dual intersection space complexes}

  \author{M. Agust\'in}
  \address{Marta Agust\'in: BCAM,  Basque Center for Applied Mathematics, Mazarredo 14, 48009 Bilbao, Basque Country, Spain}
  \email{martaav22@gmail.com}

  \author{J.T. Essig}
  \address{J. Timo Essig: Karlsruhe Institute of Technology, Engler-Bunte-Ring 1b, 76131 Karlsruhe, Germany,\href{https://orcid.org/0000-0002-9975-8721}{\includegraphics[scale=0.5]{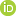}}}
  \email{timo.essig@kit.edu}

  \author{J. Fern\'andez de Bobadilla}
\address{Javier Fern\'andez de Bobadilla:  
(1) IKERBASQUE, Basque Foundation for Science, Maria Diaz de Haro 3, 48013, 
    Bilbao, Basque Country, Spain;
(2) BCAM,  Basque Center for Applied Mathematics, Mazarredo 14, 48009 Bilbao, 
Basque Country, Spain; 
(3) Academic Colaborator at UPV/EHU.} 
\email{jbobadilla@bcamath.org}

  \date{25-5-2020}

\thanks{M.A. and J.F.B. were supported by ERCEA 615655 NMST Consolidator Grant, MINECO by the project 
reference MTM2016-76868-C2-1-P (UCM), by the Basque Government through the BERC 2018-2021 program and Gobierno Vasco Grant IT1094-16, by the Spanish Ministry of Science, Innovation and Universities: BCAM Severo Ochoa accreditation SEV-2017-0718.
J.T.E. was supported by the Canon Foundation and wants to thank Hokkaido University and Prof. T. Ohmoto for their hospitality during his stay there.}

\begin{abstract}
  In this article, we prove that there is a canonical Verdier self-dual intersection space sheaf complex for the middle perversity on Witt spaces that admit compatible trivializations for their link bundles, for example toric varieties. If the space is an algebraic variety our construction takes place in the category of mixed Hodge modules. We obtain an intersection space cohomology theory, satisfying Poincar\'e duality, valid for a class of pseudomanifolds with arbitrary depth stratifications. The main new ingredient is the category of K\"unneth complexes; these are cohomologically constructuble complexes with respect to a fixed stratification, together with additional data, which codifies triviality structures along the strata. In analogy to what Goreski and McPherson showed for intersection homology complexes, we prove that there are unique K\"unneth complexes that satisfy the axioms for intersection space complexes introduced by the first and third author. This uniqueness implies the duality statements in the same scheme as in Goreski and McPherson theory.
\end{abstract} 
\maketitle

\section{Introduction}
Intersection spaces were introduced by Banagl as a Poincar\'e duality homology theory for topological pseudomanifolds
which is an alternative to Goreski and MacPherson intersection homology. A very interesting and motivating feature of the theory is a computation of Banagl, comparing the betti numbers of the intersection homology and homology of intersection spaces for the singular variety appearing in the conifold transition: it hints that homology of intersection spaces is mirror symmetric to intersection homology~(see~\cite{Ban1}). The reader may consult~\cite{BaMax} or~\cite{Ess2} for surveys on the theory. 

The idea of intersection spaces was sketched for the first time in~\cite{Ban1}, and was fully developed for spaces with isolated singularities in~\cite{Ban10}. The construction of intersection spaces is not possible for general pseudomanifolds with non-isolated singularities (several counterexamples exist in the literature), and it is an open problem to determine when they exist and when they yield a homology theory with Poincar\'e duality. 
To understand the connection between intersection space homology and intersection homology or, more general, the connection with pseudomanifolds per se, an approach which is applicable to a much bigger class of spaces is needed. Such an approach is also required for the generalization Banagl's mirror symmetry computation for conifold transitions.

Before we present our main results, let us put them into their context within the intersection spaces theory. 
In \cite{ABIS}, the first and third author proposed a method to approach intersection spaces and their homology for pseudomanifolds with arbitrary stratification depth. 
Before that, spaces of stratification depth greater than one had only been considered by Banagl and the second author in \cite{Ban10,Ban2,Ess} for special types of link fibrations and a low number of strata.

The main novelty of \cite{ABIS} is the following realization: A spatial modifiction that generalizes the previous constructions and is valid for arbitrary stratification depth needs to be implemented for pairs of spaces --modifying the pseudomanifold and its singular set at the same time.
This is done by an obstruction theory type construction, which eventually results in a pair of spaces that are natural candidates to generalize intersection spaces. It is proved that if the link fibrations of the pseudomanifold are trivial and the trivializations of these fibrations verify some compatibility conditions, a canonical intersection space pair associated with the system of trivializations exists. Note, that we explain why toric varieties satisfy these triviality properties in the present article (this was claimed without proof in~\cite{ABIS}).

A second set of results in \cite{ABIS} is the development of a sheaf theoretic approach to intersection space pairs similar to the one of Goreski and MacPherson for intersection homology in \cite{GorMP}. A set of properties, which mimic the axioms for intersection homology complexes, are introduced. Associated with any intersection space, there is a unique sheaf complex satisying those properties, a so called intersection space complex. The homology of the intersection space pair coincides with the hypercohomology of the intersection space complex. Moreover, an obstruction theory for the existence and uniqueness of intersection space complexes is given. It turns out that intersection space complexes --if they exist-- are not necessarily unique.

Although the shifted Verdier dual of an intersection space complex is an intersection space complex as well, this cannot be used to prove Poincar\'e duality for the hypercohomology of intersection space complexes: Since the complexes are not unique, one does not automatically get self-dual complexes as in~\cite{GorMP}. 
In this paper we show how to overcome these difficulties for a class of pseudomanifolds that includes complex toric varieties --the class of psedudomanifolds with trivial link bundles and suitable compatibility conditions for the trivializations. Its precise definition can be found in Section~\ref{sec:trivi}, and coincides with Definition~3.13 of~\cite{ABIS} (re-written in a way which is convenient for our purposes).
The main idea of is to enrich the category and axioms of intersection space complexes in order to be able to obtain uniqueness characterization results leading to Poincar\'e duality in a similar way than in~\cite{GorMP} for intersection homology. Let us give an overview of our method:
Fix a pseudomanifold $X$ and a compatible system of trivializations $(\phi_1,...,\phi_k)$ of the link bundles. We introduce the category of K\"unneth complexes (see Definitions~\ref{def:kunethstr},~\ref{def:kunnethmor}). A K\"unneth complex is, roughly speaking, an element of $\mathcal{D}^b_{cc}(X)$ enriched with trivializations of the restriction of the complex to the link bundles, which are induced by the compatible system of trivializations $(\phi_1,...,\phi_k)$. K\"unneth morphims respect such trivializations. 

We develop the basic properties of the category of K\"unneth complexes and prove that, for any perversity, there exists a unique intersection space complex $\IS (X)$ that is a K\"unneth complex with respect to the compatible system of trivializations $(\phi_1,...,\phi_k)$ (see Propositions~\ref{prop:KuennethIS_existence} and \ref{prop:KuennethIS_uniqueness}). We call it the {K\"unneth intersection space complex of perversity $\pp$} for the compatible system of trivializations $(\phi_{1},...,\phi_k)$. We prove that the Verdier dual of $\IS (X)$ is the  K\"unneth intersection space complex of the complementary perversity $\qq$ in Theorem~\ref{th:main1}. Moreover, the intersection space complex associated with the intersection space pair $(I^{\pp}X,I^{\pp}X_{d-2})$ is $\IS(X)$ (see Remark~\ref{rem:spacetosheaf}). If the stratification is by complex algebraic varieties, then the K\"unneth intersection space complex of perversity $\pp$ lives in the derived category of mixed Hodge modules. As a summary:
\begin{thm}[Main Theorem]
\label{thm:intro}
Let $X$ be a topological stratified pseudomanifold of dimension $d$ with a compatible system of trivializations $(\phi_1,...,\phi_k)$. 
\begin{itemize}
 \item There is a unique K\"unneth intersection space complex for any perversity.
 \item The Verdier dual of the K\"unneth intersection space complex for perversity $\pp$ shifted by dimension $d$, coincides with the K\"unneth intersection space complex for the complementary perversity.
 \item If $X$ is a Witt space the K\"unneth intersection space complex for the intermediate perversity is self-dual: $\mathcal{D}IS^\bullet_{\bar{m}}(X)[-d]\cong IS^\bullet_{\bar{m}}(X)$. 
 \item With the notations above we have Poincar\'e duality isomorphisms $H^k(X,\IS(X))\cong H^{d-k}(X,\ISq(X))^\vee$.
 \item Let $(I^{\pp}X,I^{\pp}X_{d-2})$ be the intersecion space pair associated with the system of trivializations. There is a Poincar\'e duality isomorphism  $H^k(I^{\pp}X,I^{\pp}X_{d-2};\mathbb{Q})\cong H^{d-k}(I^{\qq}X,I^{\qq}X_{d-2};\mathbb{Q})^\vee$. 
 \item If the stratification is by complex algebraic varieties the K\"unneth intersection space complex of perversity $\pp$ lives in the derived category of mixed Hodge modules, and is unique in this category. Consequently, Poincate duality is an isomorphism of mixed Hodge structures.
\end{itemize}
\end{thm}

In the case of toric varieties we provide at least two canonical stratifications with a compatible system of trivializations: the stratification by orbits and the stratification by orbits contained in the singular locus.

\begin{cor}
Toric varieties, with respect to the stratifications mentioned above, have canonical intersection space complexes in the derived category of mixed Hodge modules for any perversity. Therefore their hypercohomology is a Mixed Hodge structure, which satisfies Poincar\'e duality in the case of the middle perversity. 
\end{cor}

 While intersection homology only depends on the topological type of the pseudomanifold, intersection space complexes depend on the stratification. Even when the obstruction theory procedure of~\cite{ABIS} can be carried out until the end and the intersection space complex exists, it is unique only in rare cases (the groups obstructing it are generally non-trivial). Therefore, Verdier self-dual intersection space complexes for the middle perversity are pretty rare. We see the results of this paper as a shift in point of view: the underlying geometric object of a intersection space complex should be a stratified pseudomanifold with a {\em prescribed structure on the link bundles}, and once this structure is fixed properly, one should obtain a unique intersection space complex compatible with it. In this paper we take the first step in this direction by fixing triviality structures. It is an interesting direction to determine what kind of structures can be considered in the case of non trivial link bundles. 
 
 Another direction of development we would like to point out: compute the Hodge numbers of the canonical intersection space homologies for the stratifications of toric varieties introduced above.

\section{Notation}\label{section:notation}
In this section we gather terminology and notation that we use throughout the paper.
For a stratified pseudomanifold $X$ with filtration $X_d\supset X_{d-c_1}\supset ....\supset X_{d-c_k}$ we explain compatible systems of tubular neighborhoods and compatible systems of trivializations $(\phi_{1},...,\phi_k)$ in Section \ref{sec:tubular}. 
Let $\iota_l:X\setminus X_{d-c_l}\to X$, $j_l:X_{d-c_l}\to X$ and $\theta_l:X_{d-c_l}\setminus X_{d-c_{l+1}}\to X$ be the open, closed and locally closed inclusions of singular strata and their complements respectively. For any index $i$ let $\iota^i_l:Z^i_{c_i}\setminus Z^i_{c_i-c_l}\to Z^i_{c_i}$, $j^i_l: Z^i_{c_i-c_l}\to Z^i_{c_i}$ and  $\theta^i_l:Z^i_{c_i-c_l}\setminus Z^i_{c_i-c_{l+1}}\to Z^i_{c_i}$ be the open, closed and locally closed inclusions of the singular strata of the pseudomanifold $Z_{c_i}^i$, which belongs to the tubular neighbourhood of the stratum $X_{d-c_i}$ via $\phi_i$ (see Section~\ref{sec:tubular}). Observe that if $i<l$ then $\iota^i_l$ is the identity and $j^i_l$ is the empty inclusion.
To match our notation with the literature, we also denote the inclusion of the regular part $\iota_1: U := X \setminus X_{d-{c_1}} \hookrightarrow X$ by $i = \iota_1$ and the inclusion of the singular set $j_1: X_{d-c_1} \hookrightarrow X$ by $j=j_1.$

We work in the derived category of cohomologically constructible bounded complexes of sheaves $\Dbc$. By a slight abuse of notation, we denote right derived functors in this category by the same symbol as the functors that induce them. For example, if $i: U \hookrightarrow X$ is an open inclusion, we write $i_*$ for $Ri_*$, the derived functor of the direct image functor $i_*$.
For a map $\lambda$ of (sheaf) complexes, we denote its usual mapping cone by $C(\lambda).$

\section{Systems of trivializations}
\label{sec:trivi}
Let $X_d\supset X_{d-c_1}\supset ....\supset X_{d-c_k}$ be a stratified topological pseudomanifold (see Definition 3.4 of~\cite{ABIS}) of dimension $d$, where $X_{d-c_i}\setminus X_{d-c_{i+1}}$ is the stratum of codimension $c_i$. In this section we recall the notion of a system of trivializations. A stratified pseudomanifold admits a system of trivializations if all the strata admit tubular neighborhoods that are trivial fibre bundles, and the trivializations satisfy certain compatibility relations. 
For the rest of this paper, we assume that the strata $X_{d-c_i}\setminus X_{d-c_{i+1}}$ are connected. This assumption keeps the notation in our proofs in check. The arguments in the general case are analogous, with a more complicated notation. Only the section on toric varieties is written without this assumption.

\subsection{Compatible tubular neighborhoods.}\label{sec:tubular}
First we require that our stratified topological pseudomanifold satisfies the following conditions, which are a slight weakening of the notion of Conical Structure with respect to the stratification, introduced in \cite[Definition 3.5]{ABIS}. For each $i\in\{1,...,k\}$ there is a closed neighborhood $TX_{d-c_i}$ of $X_{d-c_i}\setminus X_{d-c_{i+1}}$ in $X\setminus X_{d-c_{i+1}}$ with the following properties:
\begin{itemize}
\item Let $\overline{TX_{d-c_i}}$ be the closure of $TX_{d-c_i}$ in $X$. There is a locally trivial fibration $\pi_i$ of $2i$-tuples of
spaces with total space
\[ (\overline{TX_{d-c_i}}\setminus X_{d-c_{i+1}})\cap (X,\overline{TX_{d-c_1}}, X_{d-c_1},..., \overline{TX_{d-c_{i-1}}}, X_{d-c_{i-1}},X_{d-c_{i}} ) \]
and base $X_{d-c_i}\setminus X_{d-c_{i+1}}.$
Its fibre is 
$$(Z^i_{c_i},\overline{TZ^i_{c_i-c_1}},Z^i_{c_i-c_1},...,\overline{TZ^i_{c_i-c_{i-1}}}, Z^i_{c_i-c_{i-1}},Z^i_0),$$
where $Z^i_{c_i}\supset Z^i_{c_i-c_1}\supset Z^i_{c_i-c_{i-1}}\supset Z^i_0$ is a topological pseudomanifold of dimension $c_i$ and $0$-dimensional smallest stratum equal to a point, which is a cone (with vertex $Z^i_0$) over a compact pseudomanifold $L^i$ of dimension $c_i-1$, called the transversal link of the stratum $X_{d-c_i}\setminus X_{d-c_{i+1}}$. For each $j\leq i-1$ the subset $TZ^i_{c_i-c_{j}}$ is a tubular neighborhood around the $c_j$-codimensional stratum of $Z^i_{c_i-c_j}$.

\item For any $i_1<i_2$ we have the equality 
$$(\overline{TX_{d-c_{i_1}}}\setminus X_{d-c_{i_1+1}})\cap \overline{TX_{d-c_{i_2}}}=\pi_{i_1}^{-1}((X_{d-c_{i_1}}\setminus X_{d-c_{i_1+1}}))\cap \overline{TX_{d-c_{i_2}}}).$$ Moreover, the compatibility
\begin{equation} 
\label{eq:comp1}
\begin{tikzcd}[column sep=small]
(\overline{TX_{d-c_{i_1}}}\setminus X_{d-c_{i_1+1}})\cap \overline{TX_{d-c_{i_2}}} \ar{dr}{\pi_{i_2}|} \ar{r}{\pi_{i_1}|} &  (X_{d-c_{i_1}}\setminus X_{d-c_{i_1+1}}))\cap \overline{TX_{d-c_{i_2}}} \ar{d}{\pi_{i_2}|} \\
\ & X_{d-c_{i_2}} \setminus X_{d - c_{i_2 + 1}} 
\end{tikzcd}
\end{equation}
is satisfied.
\end{itemize}

\begin{mydef}[Compatible system of neighborhoods] ~\\
\label{def:compatneigh}
Given a topological pseudomanifold $X_d\supset X_{d-c_1}\supset ....\supset X_{d-c_k}$, a system of neighborhoods $TX_{d-c_i}$ of $X_{d-c_i}\setminus X_{d-c_{i+1}}$ in $X\setminus X_{d-c_{i+1}}$, and fibrations $\pi_i$ for $i\in\{1,...,k\}$ satisfying the properties above is a {\em compatible system of tubular neighborhoods}.
\end{mydef}

\subsection{Systems of trivializations}

Let $X_d\supset X_{d-c_1}\supset ....\supset X_{d-c_k}$ be a pseudomanifold that admits a compatible system of tubular neighborhoods as in~\ref{sec:tubular}. 

\begin{figure}[h]
\centering
\begin{tikzpicture}[scale=0.9]
  \path[use as bounding box] (-3cm,3cm) rectangle (5cm,-3cm);
  \clip (0,0) circle (5cm);
   \begin{scope}
      \path[opacity=0]
      (0,0) 
      edge [bend right=40] 
      coordinate[pos=1] (Ul)
      coordinate[pos=0] (Ol)
      (-30:16cm)
      (-30:16cm) 
      coordinate[pos=1] (Or)
      coordinate[pos=0] (Ur)
      (0,0);
      \path[opacity=0]
      (0,0) 
      edge [bend left=40] 
      coordinate[pos=1] (ul)
      coordinate[pos=0] (ol)
      (-30:16cm)
      (-30:16cm) 
      coordinate[pos=1] (or)
      coordinate[pos=0] (ur)
      (0,0);
        \end{scope}
  \begin{scope}
  \foreach \radius in {0.5, 1, 1.5,...,3}
  \filldraw[draw=none,fill opacity=0.15,fill=blue!70!black]  (0,0) circle (\radius cm);
  \end{scope}
  \draw node[thick, blue!70!black] at (265:2.25cm) {$T$};
  \begin{scope}
	\foreach \angle in {-40, -35, -30,...,40}
	\filldraw[draw=none,fill opacity=0.15,fill=yellow!60!brown,bend right=\angle] (0,0) to (0:16cm);
  \end{scope}
  \draw        node[thick, yellow!55!red] at (12.5:4.5cm) {$T'$};
  \begin{scope}
    \clip (0,0) circle (3cm);
\draw[very thick,opacity=0.5, blue!50!black,bend right=25] (0:16cm) to (0,0);
\draw[very thick,opacity=0.5, blue!50!black,->] (23.2:1.7cm) to (23.2:1.5cm);
\foreach \angle in {75.714285714,117.428571429,159.142857142,200.857142856,242.57142857,284.285714284}
  \draw[very thick,opacity=0.5, blue!50!black,->] (\angle:16cm) to (\angle:1.5cm);
\foreach \angle in {75.714285714,117.428571429,159.142857142,200.857142856,242.57142857,284.285714284}
  \draw[very thick,opacity=0.5, blue!50!black] (\angle:16cm) to (0,0);
\draw[very thick,opacity=0.5, blue!50!black,bend left=25] (0:16cm) to (0,0);
\draw[very thick,opacity=0.5, blue!50!black,->] (-23.2:1.7cm) to (-23.2:1.5cm);
\draw node[thick,blue!50!black] at (207.5:2.25cm) {$\pi$};
  \end{scope}
    \begin{scope}
      \path[opacity=0]
      (0,0) 
      edge [bend right=40] 
      coordinate[pos=1] (Ul)
      coordinate[pos=0] (Ol)
      (0:16cm)
      (0:16cm) 
      coordinate[pos=1] (Or)
      coordinate[pos=0] (Ur)
      (0,0);
      \path[clip]
      (Ol) to [bend right=40] (Ul) to (Ur) to (Or) to cycle;
      \draw[yellow!55!red,opacity=0.5, very thick, opacity=0.7, ->] (300:1) arc (300:360:1cm);
      \draw[yellow!55!red,opacity=0.5, very thick, opacity=0.7,->] (300:2) arc  (300:360:2cm);
      \draw[yellow!55!red,opacity=0.5, very thick, opacity=0.7, ->] (300:3) arc (300:360:3cm);
      \draw[yellow!55!red,opacity=0.5, very thick, opacity=0.7,->] (300:4) arc  (300:360:4cm) node[pos=0.6, right,thick] {$\pi'$};
    \end{scope}
    \begin{scope}
      \path[opacity=0]
      (0,0) 
      edge [bend left=40] 
      coordinate[pos=1] (Ul)
      coordinate[pos=0] (Ol)
      (0:16cm)
      (0:16cm) 
      coordinate[pos=1] (Or)
      coordinate[pos=0] (Ur)
      (0,0);
      \path[clip]
      (Ol) to [bend left=40] (Ul) to (Ur) to (Or) to cycle;
      \draw[yellow!55!red, very thick, opacity=0.7, ->] (40:1) arc (40:0:1cm);
      \draw[yellow!55!red, very thick, opacity=0.7,->]  (40:2) arc (40:0:2cm);
      \draw[yellow!55!red, very thick, opacity=0.7, ->] (40:3) arc (40:0:3cm);
      \draw[yellow!55!red, very thick, opacity=0.7,->]  (40:4) arc (40:0:4cm);
    \end{scope}
  \draw[very thick,red!50!brown] 
  	(0,0) -- (0:8cm)
       node[below right,red!50!brown] at (0:4.3cm) {$S'$};
	\begin{scope}
	\clip (0,0) circle (3cm);
	\foreach \angle in {-40, -35, -30,...,40}
	\filldraw[draw=none,fill opacity=0.2,fill=green!50!black,bend right=\angle] (0,0) to (0:16cm);
	\node[green!70!blue!30!black] (A) at (-10:2.5) {$A$};
	\end{scope}
  \node[draw,circle,red!50!black,inner sep=2pt, fill] (X0) at (0,0) {};
  \node[below left,red!50!black] (X0) at (0,0) {$S$};
  \draw node at (315:3.5cm) {$X$};
\end{tikzpicture}
\caption{The intersection of tubular neighbourhoods of strata}
\label{fig:StrataTubes}
\end{figure}
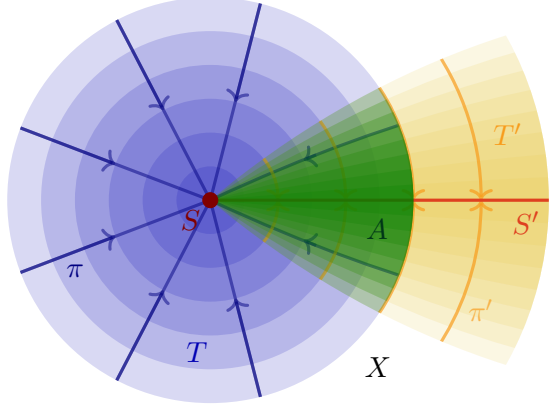
\begin{mydef}[System of trivializations] ~\\
\label{def:sistriv}
A {\em system of trivializations} is a tuple $(\phi_1,\ldots, \phi_k)$ of homeomorphisms
$\phi_i$ of $2i$-tuples from 
$$(\overline{TX_{d-c_i}}\setminus X_{d-c_{i+1}})\cap (X,\overline{TX_{d-c_1}}, X_{d-c_1},..., \overline{TX_{d-c_{i-1}}}, X_{d-c_{i-1}},X_{d-c_{i}})$$ 
to
\begin{equation}
\label{eq:prodstructure}
(Z^i_{c_i},\overline{TZ^i_{c_i-c_1}},Z^i_{c_i-c_1},...,\overline{TZ^i_{c_i-c_{i-1}}}, Z^i_{c_i-c_{i-1}},Z^i_0)\times (X_{d-c_i}\setminus X_{d-c_{i+1}})
\end{equation}
such that $\phi_i$ composed with the projection to the second factor equals $\pi_i$.
\end{mydef}

Let $i_1<i_2$ and consider the intersection 
\begin{equation}
\label{eq:defA}
A:=(\overline{TX_{d-c_{i_1}}}\setminus X_{d-c_{i_1+1}})\cap \overline{TX_{d-c_{i_2}}}.
\end{equation}
The region $A$ is illustrated in Figure \ref{fig:StrataTubes} as the green surface. The tube $\overline{TX_{d-c_{i_2}}}$ of the lower dimensional stratum is labeled $T$ in the figure and highlighted in blue, whereas restricted tube $\overline{TX_{d-c_{i_1}}}\setminus X_{d-c_{i_1+1}}$ is labeled $T'$ and highlighted in yellow. The stratum $X_{d-c_{i_2}}\setminus X_{d-c_{i_2+1}}$ is called $S$ in the picture and the stratum $X_{d-c_{i_1}}\setminus X_{d-c_{i_1+1}}$ is called $S'$.

Using the trivialization $\phi_{i_2}|_A$ we have a homeomorphism 
\begin{equation}
\label{eq:triv1}
\phi_{i_2}|_A:A\to (\overline{TZ^{i_2}_{c_{i_2}-c_{i_1}}}\setminus Z^{i_2}_{c_{i_2}-c_{i_1+1}})\times (X_{d-c_{i_2}}\setminus X_{d-c_{i_2+1}})
\end{equation}
Using the trivialization for $\phi_{i_1}|_A$ we have a homeomorphism
\begin{equation}
\label{eq:triv2}
\phi_{i_1}|_A:A\to Z^{i_1}_{c_{i_1}}\times ((X_{d-c_{i_1}}\setminus X_{d-c_{i_1+1}})\cap \overline{TX_{d-c_{i_2}}}).
\end{equation}
Using the trivilization $\phi_{i_2}$ restricted to $(X_{d-c_{i_1}}\setminus X_{d-c_{i_1+1}})\cap \overline{TX_{d-c_{i_2}}}$ we obtain a homeomorphism
\begin{equation}
\label{eq:triv3}
\phi_{i_2}|:(X_{d-c_{i_1}}\setminus X_{d-c_{i_1+1}})\cap \overline{TX_{d-c_{i_2}}}\to Z^{i_2}_{c_{i_2}-c_{i_1}}\times (X_{d-c_{i_2}}\setminus X_{d-c_{i_2+1}}).
\end{equation}
Combining the homeomorphisms (\ref{eq:triv1}), (\ref{eq:triv2}) and (\ref{eq:triv3}), and using the compatibility (\ref{eq:comp1}), we obtain a homeomorphism
\begin{equation}
\label{eq:comparison}
\begin{tikzcd}
Z^{i_1}_{c_{i_1}}\times Z^{i_2}_{c_{i_2}-c_{i_1}}\times (X_{d-c_{i_2}}\setminus X_{d-c_{i_2+1}}) \ar{d}{\psi_{i_1,i_2} ~= ~(\phi_{i_1}|)^{-1} \, \circ \,  \id_{Z_{c_{i_1}}^{i_1}} \, \circ \, (\phi_{i_2}|)^{-1}} \\
(\overline{TZ^{i_2}_{c_{i_2}-c_{i_1}}}\setminus Z^{i_2}_{c_{i_2}-c_{i_1+1}})\times (X_{d-c_{i_2}}\setminus X_{d-c_{i_2+1}}) 
\end{tikzcd}
\end{equation}
of the form 
\begin{equation}
\label{eq:alphaz}
\psi_{i_1,i_2}(x,y,z)=(\alpha_z(x,y),z)
\end{equation}
where $x$, $y$ and $z$ are points in each of the factors of the source and $\alpha_z$ is a homeomorphism that may depend on $z$.

The last equation has the following meaning. For any $i_2$ and any $z\in X_{d-c_{i_2}}\setminus X_{d-c_{i_2+1}}$ the stratified pseudomanifold $$Z^{i_2}_{c_{i_2}}\supset Z^{i_2}_{c_{i_2}-c_1}\supset Z^{i_2}_{c_{i_2}-c_{i_2-1}}\supset Z^{i_2}_0,$$ 
together with the compatible system of tubular neighborhoods $\overline{TZ^{i_2}_{c_{i_2}-c_{i_1}}}$ for $1\leq i_1\leq i_2$, has a system of trivializations. In fact for any $i_1\leq i_2$ the trivialization is given by the inverse of the homeomorphism $\alpha_z$ defined in~(\ref{eq:alphaz}). 

\begin{mydef}[Compatible system of trivializations] ~\\
\label{def:compatsistriv}
A {\em compatible system of trivializations} of $X_d\supset X_{d-c_1}\supset ....\supset X_{d-c_k}$ is a system such that $\alpha_z$ is independent of $z$ for any $i_1<i_2$.
\end{mydef}
\begin{rem}
\label{rem:compatIS}
Given a compatible system of trivializations of $X_d\supset X_{d-c_1}\supset ....\supset X_{d-c_k}$ and a perversity $\pp$, we have the following.
\begin{enumerate}
\item A compatible system of trivializations is inherited for each of the pseudomanifolds  $Z^i_{c_i}\supset Z^i_{c_i-c_1}\supset Z^i_{c_i-c_{i-1}}\supset Z^i_0$.
\item The construction of~\cite{ABIS} provides a (non-necessarily unique) intersection space pair $(I^{\pp}X,I^{\pp}X_{d-2})$ (see Definition 3.27~of~\cite{ABIS}).
\end{enumerate}
\end{rem}
As a first example for compatible systems of trivializations, we explain how to extend such a system from a pseudomanifold $X$ to its cone $\cone{X}.$
\begin{ex}[Trivializations on a cone]
If $X_d\supset X_{d-c_1}\supset ....\supset X_{d-c_k}$ is a pseudomanifold with a compatible system of trivializations $(\phi_1, \ldots, \phi_k)$ and $Y = \cone{X}$ with vertex $\left\{ v \right\},$ then $Y$ has a filtration $Y = Y_{d+1} \supset Y_{d +1 - c_1} \supset \ldots \supset Y_{d+1 - c_k} \supset Y_0 = \left\{ v \right\}$ with $Y_{d+1 - c_i} = \cone{X_{d-c_i}}$. Moreover, $Y$ has a compatible system of neighbourhouds $TY_{d+1 - c_i} := \cone{TX_{d-c_i}}$ and $TY_0 := Y$  with
\[
\overline{TY_{d+1-c_i}} \setminus Y_{d+1 - c_{i+1}} = \overline{TX_{d-c_i}} \setminus X_{d - c_{i+1}} \times (0,1),
\]
 and the homeomorphisms $\phi_i \times \id$ from $\overline{TY_{d+1-c_i}} \setminus Y_{d+1 - c_{i+1}}$ to 
\[ Z_{c_i}^i \times \left( Y_{d+1 - c_i} \setminus Y_{d+1 - c_{i+1}} \right) = Z_{c_i}^i \times \left( X_{d-c_i} \setminus X_{d - c_{i+1}} \times (0,1) \right) \]
and $\phi_0 = \id: \overline{TY_0} \xrightarrow{\cong} \cone{X} \times \left\{ v \right\}$
give rise to a compatible system of trivializations $(\phi_1 \times \id, \ldots, \phi_k \times \id, \phi_0)$ of $Y = Y_{d+1} \supset Y_{d +1 - c_1} \supset \ldots \supset Y_{d+1 - c_k} \supset Y_0$.
\end{ex}
We present an important class of pseudomanifolds that admit compatible systems of trivializations in the following section: toric varieties. The stratifications that come with such a compatible system of trivializations are inherited from the torus action of such a variety.

\section{Trivialization structure for Toric Varieties}
\label{sec:torictriv}
In this section, we expound how the group action on toric varieties induces a compatible system of trivializations for stratifications which are compatible with the torus action in the sense explained below.

Recall that an algebraic normal variety $X$ is called an $n$-dimensional toric variety if there is group action $\TT \times X \to X$ of the complex torus $\TT \cong (\C^*)^n$ which is almost transitive and effective. As a general reference we use~\cite{Fu}. 

\begin{mydef}
\label{def:toriccompatible}
Let $X$ be complex toric variety. A pseudomanifold structure $X=X_{2d}\supset X_{2d-2c_1}\supset ....\supset X_{2d-2c_k}$ is compatible with the torus action if each non-open stratum is an orbit of the torus action.
\end{mydef}

\begin{rem}
Given a complex toric variety, there are two canonical pseudomanifold structures which are compatible with the torus action on it. The first one has all the orbits as strata. The strata of the second one are the non-singular loci as well as the orbits which are contained in the singular set.
\end{rem}

\subsection{Reminder on toric geometry facts}
A toric variety of complex dimension $d$ is given by a fan $\Delta$ in $\mathbb{Z}^d$, which is a set of rational strongly convex polyhedral cones in $\mathbb{R}^d=\mathbb{Z}^d\otimes\mathbb{R}$ such that each face of a cone in $\Delta$ is also in $\Delta$, and such that any two cones intersect in a common face. Let $X^\Delta_{2d}$ be the toric variety associated with $\Delta$. The orbits of the group action are in a $1:1$ correspondence with the faces of $\Delta$. Given any face $\sigma$ of $\Delta$, the complex dimension of the orbit $O^\Delta_{\sigma}$ associated with it equals $d-\dim(\sigma)$, where $\dim(\sigma)$ denotes the real dimension of the affine subspace spanned by $\sigma$ in $\mathbb{R}^d$. 
Moreover, there is an affine open subset $U^\Delta_{\sigma}$ of $X^\Delta_{2d}$, which is the affine toric variety associated with the polyhedral cone $\sigma$, and $O^\Delta_{\sigma}$ is the only closed orbit of $U^\Delta_{\sigma}$. We denote by $I_\sigma$ the stabilizer of any point of $O^\Delta_{\sigma}$ by the torus action. It is independent of the point in $O^\Delta_\sigma$, because the torus is abelian, and it is isomorphic to $(\C^*)^{\dim(\sigma)}$.  

Let us fix a generic complete flag 
\begin{equation}
\label{eq:genflag}
\mathbb{T}^1\subset\mathbb{T}^2\subset ...\subset\mathbb{T}^d=(\C^*)^d
\end{equation}
of subtori of $(\C^*)^d$. Generic means that $\mathbb{T}^i$ acts transitively on every orbit of dimension $i$. This is equivalent to the fact that for any face $\sigma$, we have the equality $\mathbb{T}^{d-\dim(\sigma)}\cap I_\sigma=\{(1,...,1)\}$, and hence that we have the direct sum splitting 
\begin{equation}
\label{eq:dirsum}
(\C^*)^d=\mathbb{T}^{d-\dim(\sigma)}\oplus I_\sigma.
\end{equation}
We denote by $\rho^\Delta_\sigma:(\C^*)^d\to \mathbb{T}^{d-\dim(\sigma)}$ the corresponding projection to the first factor.

Any face $\sigma$ can be seen as a strongly convex rational polyhedral cone in $V_\sigma$, where $V_\sigma$ is the vector subspace spanned by $\sigma$ (which is defined over $\mathbb{Q}$, and where we consider the lattice given by the intersection of $\mathbb{Z}^d$ with $V_{\sigma}$). We can then associate with $\sigma$ a $\dim(\sigma)$-dimensional toric variety $X^\sigma_{2\dim(\sigma)}$, which is a compactification of the torus $I_\sigma$. 

The group $\mathbb{T}^{d-\dim{\sigma}}$ acts freely on $U^\Delta_\sigma$, and the action induces a product structure  
\begin{equation}
\label{eq:dirproduct}
U^\Delta_\sigma=\mathbb{T}^{d-\dim(\sigma)} \times X^\sigma_{2\dim(\sigma)},
\end{equation}
which extends the direct sum decomposition~(\ref{eq:dirsum}). We denote by $\bar\rho^\Delta_\sigma:(\C^*)^d\to \mathbb{T}^{d-\dim(\sigma)}$ the corresponding projection to the first factor. Notice that $\mathbb{T}^{\dim(\sigma)}$ can be identified with the orbit $O^\Delta_\sigma$. 

Given any face $\sigma$, for any $i\in\{1,... ,\dim(\sigma)\}$ we define the torus $\mathbb{T}_\sigma^i:=\mathbb{T}^{i-d+\dim(\sigma)}\cap I_\sigma$ . We obtain a flag of subtori 
\begin{equation}
\label{eq:genflagsigma}
\mathbb{T}_\sigma^1\subset\mathbb{T}_\sigma^2\subset ...\subset\mathbb{T}_\sigma^{\dim(\sigma)}.
\end{equation}

Given two faces $\sigma,\tau$ of the fan $\Delta$, the inclusion $\tau\subset\sigma$ is equivalent to the inclusion $\overline{O}^\Delta_\tau\supset O^\Delta_\sigma$. Assume the inclusion holds. Then we have that $I_\tau$ is a subgroup of $I_\sigma$. The flag~(\ref{eq:genflagsigma}) satisfies the equality $\mathbb{T}_\sigma^{\dim(\sigma)-\dim(\tau)}\cap I_\tau=\{(1,...,1)\}$. Let $U^\sigma_\tau$ be the affine open subset of the toric variety $X^\sigma_{2\dim(\sigma)}$ corresponding with the face $\tau$. The action of $\mathbb{T}_\sigma^{\dim(\sigma)-\dim(\tau)}$ on $U^\sigma_\tau$ is free and induces a product structure 
\begin{equation}
\label{eq:dirproductsigma}
U^\sigma_\tau=\mathbb{T}_\sigma^{\dim(\sigma)-\dim(\tau)} \times X^\tau_{2\dim(\tau)},
\end{equation}
analogously to the way in which the product structure~(\ref{eq:dirproduct}) was obtained.

Also, if we have the inclusion $\tau\subset\sigma$, then we have the inclusion $U^\Delta_\tau\subset U^\Delta_\sigma$ of affine open subsets of $X_{2d}$. The action of the group $\mathbb{T}^{d-\dim(\sigma)}$ restricts to a free action on $U^\Delta_\tau$, and the product structure~(\ref{eq:dirproduct}) induces a product structure
\begin{equation}
\label{eq:dirproductrestricted}
U^\Delta_\tau=\mathbb{T}^{d-\dim(\sigma)} \times U^\sigma_{\tau}.
\end{equation}
Combining with the product structure~(\ref{eq:dirproductsigma}) we obtain  
\begin{equation}
\label{eq:dirproductcombined}
U^\Delta_\tau=\mathbb{T}^{d-\dim(\sigma)} \times \mathbb{T}_\sigma^{\dim(\sigma)-\dim(\tau)} \times X^\tau_{2\dim(\tau)}=\mathbb{T}^{d-\dim(\tau)} \times X^\tau_{2\dim(\tau)},
\end{equation}
which is the product structure associated with the face $\tau$ in $\Delta$. This implies the following relation of projection mappings
\begin{equation}
\label{eq:relproj}
\rho^\Delta_\sigma|_{U^\Delta_\tau}=\rho^\Delta_\sigma|_{O^\Delta_\tau}\circ\rho^\Delta_{\tau}.
\end{equation}

\subsection{Toric stratified varieties and their trivialization structures}
Consider a complex toric variety $X=X_{2d}\supset X_{2d-2c_1}\supset \ldots\supset X_{2d-2c_k}$ with a pseudomanifold structure compatible with the torus action.
Then $X_{2d-2c_i}\setminus X_{2d-2c_{i+1}}$ is a disjoint union of orbits $O_\sigma$, where $\sigma$ is a face in $\Delta$ of dimension $c_i$. We build a compatible system of tubular neighborhoods and trivializations by increasing induction on the dimension $c_i$. Let $\mathcal{C}$ be the set of faces corresponding to strata. Split $\mathcal{C}=\coprod_{i=1}^k\mathcal{C}_i$, where $\mathcal{C}_i$ is the collection of faces of dimension $c_i$ whose associated orbit is a stratum. 

Given $\sigma\in\mathcal{C}_1$, the affine toric variety $X^\sigma_{2\dim(\sigma)}$ has a unique $0$-dimensional orbit. Let $Y_\sigma$ be the intersection of $X^\sigma_{2\dim(\sigma)}$ with a small closed ball centered at the $0$-dimensional orbit. Define the tubular neighborhood $TO^\Delta_\sigma$ to be the image of $\mathbb{T}^{d-c_1}\times Y_\sigma$ in $U^\Delta_\sigma$ by the product structure~(\ref{eq:dirproduct}). Hence we have a product structure
\begin{equation}
\label{eq:dirproductneigh}
TO^\Delta_\sigma=\mathbb{T}^{d-\dim(\sigma)} \times Y_\sigma=O^\Delta_\sigma \times Y_\sigma,
\end{equation}
where the first factor projection is given by the restriction $\rho^\Delta_\sigma|_{TO^\Delta_\sigma}.$ If the balls are chosen small enough, the neighborhoods $TO^\Delta_\sigma$ are mutually disjoint for varying $\sigma$ in $\mathcal{C}_1$. 

Assume that the tubular neighborhoods $TO^\Delta_\sigma$ have been constructed for every $\sigma\in\coprod_{j=1}^{i-1}\mathcal{C}_i$. 
Choose $\sigma\in\mathcal{C}_i$. As before the affine toric variety $X^\sigma_{2\dim(\sigma)}$ has a unique $0$-dimensional orbit. Let $Y^0_\sigma$ be the intersection of $X^\sigma_{2\dim(\sigma)}$ with a small closed ball centered at the $0$-dimensional orbit. Let $\tau\in \mathcal{C}_\sigma$, where $\mathcal{C}_\sigma$ are the faces in $\mathcal{C}$ contained in $\sigma$. By the induction hypothesis applied to the affine toric variety associated to the cone $\sigma$, with the stratification whose strata are indexed by $\mathcal{C}_\sigma$, there is a tubular neighborhood $TO^\sigma_\tau$ of $O^\sigma_\tau$ in $U^\sigma_\tau$, together with a product structure   
\begin{equation}
\label{eq:dirproductneigh2}
TO^\sigma_\tau=O^\sigma_\tau \times Y_\tau,
\end{equation}
where the first factor projection is given by the restriction $\rho^\sigma_\tau|_{TO^\sigma_\tau}.$ Define
\[ Y_\sigma:=Y^0_\sigma\cup \left(\bigcup_{\tau\in\mathcal{C}_\sigma}(\rho^\sigma_\tau|_{TO^\sigma_\tau})^{-1}(O^\sigma_\tau\cap Y^0_\sigma) \right). \]
Define the tubular neighborhood $TO^\Delta_\sigma$ to be the image of $\mathbb{T}^{d-c_i}\times Y_\sigma$ in $U^\Delta_\sigma$ by the product structure~(\ref{eq:dirproduct}). As before, we have a product structure
\begin{equation}
\label{eq:dirproductneigh3}
TO^\Delta_\sigma=\mathbb{T}^{d-\dim(\sigma)} \times Y_\sigma=O^\Delta_\sigma \times Y_\sigma,
\end{equation}
where the first factor projection is given by the restriction $\rho^\Delta_\sigma|_{TO^\Delta_\sigma}.$ Again, if the balls are chosen small enough, the neighborhoods $TO^\Delta_\sigma$ are mutually disjoint for varying $\sigma$ in $\mathcal{C}_i$.

By construction, the system of neighborhoods and the product structures above are a compatible system of neighborhoods and a compatible system of trivializations.

\section{K\"unneth structures}
Let $X_d\supset X_{d-c_1}\supset ....\supset X_{d-c_k}$ be a stratified pseudomanifold with a compatible system of tubular neighborhoods and a compatible system of trivializations $(\phi_{1},...,\phi_k)$. As we explained above, for any $i\leq k$ the normal slice 
$Z^i_{c_i}\supset Z^i_{c_i-c_1}\supset Z^i_{c_i-c_{i-1}}\supset Z^i_0$ is a  stratified pseudomanifold that inherits a compatible system tubular neighborhoods and a compatible system of trivializations $(\phi^i_1,...,\phi^i_i)$. For each $i_1<i$ the trivialization $\phi^i_{i_1}$ induces a product structure 
\begin{equation}
\label{eq:prodstrucres}
\phi^{i}_{i_1}:TZ^{i}_{c_{i}-c_{i_1}}\to Z^{i_1}_{c_{i_1}}\times (Z^{i}_{c_{i}-c_{i_1}}\setminus Z^{i}_{c_{i}-c_{i_1+1}}).
\end{equation}

Notice, however, that since the lowest dimensional stratum $Z^i_0$ is a point, the trivialization $\phi^i_i$ is meaningless. 

In the following, we define the K\"unneth property for a sheaf complex with respect to a compatible system of trivializations. We do this by induction on the dimension $d$. 

A sheaf complex $\KK \in \Dbc (X_d)$ satisfies the \emph{K\"unneth property} with respect to the compatible system of trivializations $(\phi_{1},...,\phi_k)$ if 
\begin{enumerate}[label=(\roman*)]
 \item for any $i\in \{1,...,k\}$ there exists a sheaf complex $\LL_i \in \Dbc (Z^i_{c_i})$ that satisfies the  \emph{K\"unneth property} with respect to the compatible system of trivializations $(\phi^i_1,...,\phi^i_i)$,  and a trivializing isomorphism 
 $$\beta_i: \KK|_{TX_{d-c_i}\setminus X_{d-c_{i+1}}}\xrightarrow{\cong} \phi_i^* (\pi^i_1)^* \LL_i,$$
 where $\pi_1$ denotes the projection to the first factor in the product structure (\ref{eq:prodstructure}).
 \item For any $i_1<i_2$ the following compatibility relation is satisfied. Consider the restriction $\KK|_A$, where $A$ is defined in~(\ref{eq:defA}). The trivializing isomorphisms $\beta_{i_1}$ and $\beta_{i_2}$ restrict to isomorphisms 
 \begin{equation}
 \label{eq:compatsheaf1}
 \beta_{i_1}|_A: \KK|_{A}  \xrightarrow{\cong} (\phi_{i_1}|_A)^* (\pi^{i_1}_1)^* \LL_{i_1},
 \end{equation}
\begin{equation}
 \label{eq:compatsheaf2}
 \beta_{i_2}|_A: \KK|_{A}  \xrightarrow{\cong} (\phi_{i_2}|_A)^* (\pi^{i_2}_1)^* \LL_{i_2}|_{TZ^{i_2}_{c_{i_2}-c_{i_1}}\setminus Z^{i_2}_{c_{i_2}-c_{i_1+1}}}.
\end{equation}
Since $\LL_{i_2}$ satisfies the K\"unneth property with respect to the compatible system of trivializations $(\phi^{i_2}_1,...,\phi^{i_2}_{i_2}),$ there is a trivializing isomorphism 
$$\beta^{i_2}_{i_1}:\LL_{i_2}|_{TZ^{i_2}_{c_{i_2}-c_{i_1}}\setminus Z^{i_2}_{c_{i_2}-c_{i_1+1}}}\to (\phi^{i_2}_{i_1})^*(\pi^{i_2,i_1}_1)^*\LL_{i_1},$$ 
where $\pi^{i_2,i_1}_1$ denotes the first projection in the product structure~(\ref{eq:prodstrucres}).  Together with~(\ref{eq:compatsheaf2}) this induces an isomorphism
\begin{equation}
 \label{eq:compatsheaf3}
(\phi_{i_2}|_A)^* (\pi^{i_2}_1)^*\beta^{i_2}_{i_1}\circ\beta_{i_2}|_A: \KK|_{A}  \xrightarrow{\cong} (\phi_{i_2}|_A)^* (\pi^{i_2}_1)^* (\phi^{i_2}_{i_1})^*(\pi^{i_2,i_1}_1)^*\LL_{i_1}.
\end{equation}
By the compatibility~(\ref{eq:comp1}) and the fact that the system of trivializations is compatible we have the equality 
$$(\phi_{i_1}|_A)^* (\pi^{i_1}_1)^* \LL_{i_1}=(\phi_{i_2}|_A)^* (\pi^{i_2}_1)^* (\phi^{i_2}_{i_1})^*(\pi^{i_2,i_1}_1)^*\LL_{i_1}.$$
Under this equality, the compatibility relation that has to be satisfied is the equality
\begin{equation}
\label{eq:compatisos}
\beta_{i_1}|_A=(\phi_{i_2}|_A)^* (\pi^{i_2}_1)^*\beta^{i_2}_{i_1}\circ\beta_{i_2}|_A.
\end{equation}
\end{enumerate}

\begin{rem}
 \label{rem:new}
 Observe that when $c_i=d$ the dimension of $Z^i_{c_i}$ is equal to the dimension of $X_d$, but in this case the product structure associated with the stratum is trivial, and the K\"unneth property just predicts an isomorphism $\beta_i: \KK|_{TX_{0}}\xrightarrow{\cong} \LL_i$. By this remark the definition by induction on the dimension is possible.
\end{rem}

\begin{mydef}[K\"unneth structures] ~\\
\label{def:kunethstr}
Again this definition is by induction on the dimension.
Let $X_d\supset X_{d-c_1}\supset ....\supset X_{d-c_k}$ be a stratified pseudomanifold with a compatible system of tubular neighborhoods and a compatible system of trivializations $(\phi_{1},...,\phi_k)$. Let $\KK \in \Dbc (X_d)$. A {\em K\"unneth structure} for $\KK$ is a tuple $(\LL_1, \beta_1,...,\LL_k,\beta_k)$, where each $\LL_i$ is a K\"unneth complex on $Z^i_{c_i}$ and which satisfies the K\"unneth property defined above. 

The set of data $(\KK,(\LL_1, \beta_1,...,\LL_k,\beta_k))$ is called a {\em K\"unneth complex}. If it does not lead to confusion we often denote a  K\"unneth complex by just $\KK$. 
\end{mydef}

\begin{rem}\label{rem:inducedKunnethstr}
Notice that, for any $i\in \{1,...,k\}$, a K\"unneth structure for $\KK$ induces a K\"unneth structure for $\LL_i$ for the inherited compatible system of neighborhoods and trivializations of $Z^i_{d-c_i}$. 
\end{rem}

\begin{mydef}[K\"unneth morphisms] ~\\
\label{def:kunnethmor}
Let $X_d\supset X_{d-c_1}\supset ....\supset X_{d-c_k}$ be a stratified pseudomanifold with a compatible system of tubular neighborhoods and a compatible system of trivializations $(\phi_{1},...,\phi_k)$.
A {\em K\"unneth morphism} 
$$(\KK;(\LL_1, \beta_1,...,\LL_k,\beta_k))\rightarrow (\GG;(\II_1, \beta'_1,...,\II_k,\beta'_k))$$
of K\"unneth complexes is a morphism in the derived category $\Psi: \KK \rightarrow \GG,$ and a sequence of K\"unneth morphisms $\psi_i: \LL_i \rightarrow \II_i$ such that for any $i\in\{1,...,k\}$ the diagram
\[ 
\begin{tikzcd}
\KK|_{TX_{d-c_i}\setminus X_{d-c_{i+1}}}              \ar{r}{\beta_i}[swap]{\cong} \ar{d}{\Psi} 	& (\phi_i)^* (\pi^i_1)^* \LL_1 \ar{d}{\phi_i^* \pi_1^* \psi_i} \\
\GG|_{TX_{d-c_i}\setminus X_{d-c_{i+1}}} \ar{r}{\beta'_i}[swap]{\cong} 			& (\phi_i)^* (\pi^i_1)^* \II_i
\end{tikzcd}
\]
commutes.
\end{mydef}

Observe that the definition above is by induction on the dimension, and is possible because of Remark~\ref{rem:new}.

\begin{rem}
Notice that, for any $i\in \{1,...,k\}$, by the compatibility properties stated above, a K\"unneth morphism 
$$(\KK;(\LL_1, \beta_1,...,\LL_k,\beta_k))\rightarrow (\GG;(\II_1, \beta'_1,...,\II_k,\beta'_k))$$
induces an K\"unneth morphism from $\LL_i$ to $\II_i$ for the inherited K\"unneth structures.
\end{rem}

\begin{lemma}[K\"unneth structure for mapping cones] \label{lem:cones}\ \\
Let $(\KK;(\LL_1, \beta_1,...,\LL_k,\beta_k))$ and $(\GG;(\II_1, \beta'_1,...,\II_k,\beta'_k))$ be K\"unneth complexes and let $\Psi: \KK \rightarrow \GG,$ $\psi_{i}: \LL_{i} \rightarrow \II_{i}$ provide a  K\"unneth morphism between them. Then $cone(\Psi)$ has an induced K\"unneth structure for the same system of trivializations such that all the morphisms of the triangle 
\[ \to\KK\xrightarrow{\Psi}\GG\to cone(\Psi)\xrightarrow{[1]}\]
fit into K\"unneth morphisms.

Suppose we have a commutative diagram of  K\"unneth morphisms
\[ 
\begin{tikzcd}
  (\KK_1;(\LL_{1,1}, \beta_{1,1},...,\LL_{1,k},\beta_{1,k}))            \ar{r}{\Psi_1}[swap]{} \ar{d}{\rho} 	&   (\GG_1;(\II_{1,1}, \beta'_{1,1},...,\II_{1,k},\beta'_{1,k}))                       \ar{d}{\sigma} \\
(\KK_2;(\LL_{2,1}, \beta_{2,1},...,\LL_{2,k},\beta_{2,k})) \ar{r}{\Psi_2}[swap]{} 			& (\GG_1;(\II_{2,1}, \beta'_{2,1},...,\II_{2,k},\beta'_{1,k})),
\end{tikzcd}
\]
Then the K\"unneth morphisms $\rho$ and $\sigma$ induce a K\"unneth morphism from $cone(\Psi_1)$ to $cone(\Psi_2)$, which provides a morphism of the corresponding triangles in which all arrows are K\"unneth morphisms. 
\end{lemma}
\proof
By~\cite{GelfandManin} III.3.2. we have $cone(\Psi)=\KK[1]\oplus\GG$. The K\"unneth structure for the cone is provided by the morphisms 
$$(\beta_i[1],\beta'_i):\KK[1]| \oplus \GG| \xrightarrow{\cong} \phi_i^* (\pi^i_1)^* \LL_i[1]\oplus\phi_i^* (\pi^i_1)^* \II_i,$$ 
where we restrict $\KK[1]$ and $\GG$ to $TX_{d-c_i}\setminus X_{d-c_{i+1}}.$

The second assertion is straightforward.
\endproof

We want to give some straightforward examples for K\"unneth structures. To do so, we use the notation for the inclusions of singular strata and their complements of $X$ and the $Z_{c_i}^i$ that belong to $X_{d-c_i}$ via the system of trivializations that was explained in Section \ref{section:notation}: $\iota_l:X\setminus X_{d-c_l}\to X$, $j_l:X_{d-c_l}\to X$ and $\theta_l:X_{d-c_l}\setminus X_{d-c_{l+1}}\to X$ and, for any index $i$, $\iota^i_l:Z^i_{c_i}\setminus Z^i_{c_i-c_l}\to Z^i_{c_i}$, $j^i_l: Z^i_{c_i-c_l}\to Z^i_{c_i}$ and  $\theta^i_l:Z^i_{c_i-c_l}\setminus Z^i_{c_i-c_{l+1}}\to Z^i_{c_i}$ are open, closed and locally closed inclusions. To match our notation with the literature, we also denote the inclusion of the regular part $\iota_1: U := X \setminus X_{d-c_1} \hookrightarrow X$ by $i = \iota_1$ and the inclusion of the singular set $j_1: X_{d-c_1} \hookrightarrow X$ by $j=j_1.$

\begin{ex}[K\"unneth property of the constant sheaf]\label{ex:Kunnethstr}
Let $X_d\supset X_{d-c_1}\supset ....\supset X_{d-c_k}$ be a stratified pseudomanifold with a compatible system of tubular neighborhoods and a compatible system of trivializations $(\phi_1, \ldots, \phi_k)$. As easy examples for K\"unneth structures, we look at the constant sheaf $\Q_X \in \Dbc$, the direct image ${i}_* \Q_U \in \Dbc$ of the constant sheaf on the regular part, and the sheaf $j_* j^* i_* \Q_U$. For the constant sheaf $\Q_X$, take $\LL_i := \Q_{Z^i_{c_i}}$. For $i_* \Q_U$ take $\LL_i := (\iota_1^i)_*\Q_{Z^i_{c_i}\setminus Z^i_{c_i-c_1}}$. For $j_* j^* i_* \Q_U$ take $\LL_i := (j_1^i)_* (j_1^i)^*(\iota_1^i)_*\Q_{Z^i_{c_i}\setminus Z^i_{c_i-c_1}}$. 
\end{ex}
There is a useful generalization of the example above. 
\begin{ex}
\label{ex:generalized}
Let $G$ be a abelian group. Denote by $\underline{G}_Y$ be the constant sheaf with stalk $G$ on $Y$. The complex $(\theta_l)_*\underline{G}_{X_{d-c_l}\setminus X_{d-c_{l+1}}}$ has canonical a K\"unneth structure with $\LL_i=(\theta^i_l)_*\underline{G}_{Z^i_{c_i-c_l}\setminus Z^i_{c_i-c_{l+1}}}$ induced from the system of trivializations. 
\end{ex}

\begin{prop}
\label{prop:basicpropertieskunneth}
Let $\KK$ be a constructible complex and $(\LL_1, \beta_1,...,\LL_k,\beta_k)$ be a K\"unneth structure for $\KK$. 
The following are satisfied
\begin{enumerate}[label=(\Roman*)]
 \item the complex $(\iota_l)_*(\iota_l)^*\KK$ has a K\"unneth structure given by 
 $$((\iota^1_l)_*(\iota^1_l)^*\LL_1, \gamma_1,...,(\iota^k_l)_*(\iota^k_l)^*\LL_k,\gamma_k),$$
 where $\gamma_i$ is the composition of $(\iota_l)_*(\iota_l)^*\beta_i$ and the natural isomorphism from $(\iota_l)_*(\iota_l)^*\phi_i^* (\pi^i_1)^* \LL_i$ to $\phi_i^* (\pi^i_1)^*(\iota^i_l)_*(\iota^i_l)^* \LL_i$.
\item the complex $(\iota_l)_!(\iota_l)^!\KK$ has a K\"unneth structure given as above, replacing $(\iota^i_l)_*$ and $(\iota^i_l)^*$ by $(\iota^i_l)_!$ and $(\iota^i_l)^!$ respectively.
 \item the complex $(j_l)_*(j_l)^*\KK$ has a K\"unneth structure given as in (I), replacing $(\iota^i_l)_*$ and $(\iota^i_l)^*$ by $(j^i_l)_*$ and $(j^i_l)^*$ respectively.
 \item the complex $(\theta_l)_*(\theta_l)^*\KK$ has a K\"unneth structure given as in (I), replacing $(\iota^i_l)_*$ and $(\iota^i_l)^*$ by $(\theta^i_l)_*$ and $(\theta^i_l)^*$ respectively.
 \item the adjunction morphisms $\KK \rightarrow (\iota_l)_*(\iota_l)^*\KK,$ $\LL_i \rightarrow (\iota^i_l)_*(\iota^i_l)^*\LL_i$ provide a morphism of K\"unneth structures. The same is true for the adjunctions $\KK \rightarrow (j_l)_*(j_l)^*\KK$, $\KK \rightarrow (\theta_l)_*(\theta_l)^*\KK$ and $(\iota_l)_!(\iota_l)^!\KK\to \KK$. 
\end{enumerate}
\end{prop}
\proof
Left to the reader.
\endproof

\begin{lemma}
\label{lem:trivialidad}
Let $\KK$ be a K\"unneth complex. For any stratum there are integers $a\leq b$, abelian groups $G_r$ for $a\leq r\leq b$ and a K\"unneth isomorphism
$$(\theta_{n})_*(\theta_{n})^*\KK\cong \oplus_{r=a}^{b}(\theta_{n})_*(\underline{G_r})_{X_{d-c_{n}}\setminus X_{d-c_{n+1}}}[-r],$$
where $(\underline{G_r})_{X_{d-c_{n}}\setminus X_{d-c_{n+1}}}[-r]$ denotes the constant sheaf at $X_{d-c_{n}}\setminus X_{d-c_{n+1}}$ with stalk $G_r$ shifted to degree $r$, and the K\"unneth structure considered in each direct summand on the right is defined in Example~\ref{ex:generalized}. 
\end{lemma}
\proof
The proof is an immediate consequence of Property~(i) of K\"unneth structures. It implies that $\theta_{n}^* \KK$ is isomorphic to the pullback of a sheaf complex over a point, which is just a chain complex of abelian groups. Since every chain complex is quasi-isomorphic to the direct sum of its homology groups, there is an isomorphism to a direct sum of constant sheaves
$$(\theta_{n})^*\KK\cong \oplus_{r=a}^{b}(\underline{G_r})_{X_{d-c_{n}}\setminus X_{d-c_{n+1}}}[-r],$$
and Example~\ref{ex:generalized} gives the statement.
\endproof

\section{The K\"unneth intersection space complex}\label{subs:KIScomplex}
In the following, we construct a K\"unneth intersection space complex of sheaves for any perversity, that is a K\"unneth complex that satisfies the axioms [AXS1] of \cite{ABIS}. We show that this complex is characterized (up to K\"unneth isomorphism) by a set of properties, just as for intersection homology. We use that result to show that the Verdier dual of such a sheaf complex is again a K\"unneth intersection space complex for the complementary perversity. As a consequence, for the intermediate perversity, and in the case of Witt spaces we get is self-dual complexes. This establishes Poincar\'e duality for intersection space cohomology for pseudomanifolds with compatible trivializations.

We remind axioms $[AXS1]$ for convenience of the reader. Let $X_d\supset X_{d-c_1}\supset ....\supset X_{d-c_k}$ be a stratified pseudomanifold with a compatible system of tubular neighborhoods and a compatible system of trivializations $(\phi_{1},...,\phi_k)$. Let $\pp$ be a perversity with complementary perversity $\qq$. A complex $B^\bullet\in D^b_{cc}(X_d)$ satisfies $[AXS1]$ if for any $l\leq k$
\begin{itemize}
    \item[($a$)] $B^\bullet_{|X_d\setminus X_{d-c_1}}$ is quasi-isomorphic to $\Q_{X_d\setminus X_{d-c_1}}$,
    \item[($b$)] the cohomology sheaf $\mathcal{H}^i(B^\bullet)$ is $0$ if $i \notin \{0,1,...,d\}$,
    \item[($c_l$)] $\mathcal{H}^i((j_l^* B^\bullet)|_{X_{d-c_l}\setminus X_{d-c_{l+1}}})$ is equal to $0$ if $i \leq \bar q(c_l)$,
    \item[($d_l$)] the natural morphism 
    $$\mathcal{H}^i((j_l^* B^\bullet)|_{X_{d-c_l}\setminus X_{d-c_{l+1}}}\rightarrow
    \mathcal{H}^i((j_l^* i_{l *}i_{l}^* B^\bullet)|_{X_d\setminus X_{d-c_{l+1}}})$$
    is an isomorphism if $i> \bar q(c_l)$.
\end{itemize}

\textsc{Properties $[AXKS1_{\pp}]_n$}: fix $n \in \left\{ 1, \ldots, k+1 \right\}$. We say that a K\"unneth complex 
\[ \left(\BB(X), (\BB(Z_{c_1}^1), \beta_1, \ldots, \BB(Z_{c_k}^k), \beta_k) \right) \] 
satisfies the properties $[AXKS1_{\pp}]_n$ if 
\begin{enumerate}[label=$(\alph*_n)$]
\item the restrictions 
\begin{align*}
\BB(X)&|_{X \setminus X_{d - c_n}} \in \Dbc (X \setminus X_{d - c_n}) ~\text{and} \\
\BB(Z_{c_i}^i)&|_{Z_{c_i}^i \setminus Z_{c_i - c_n}^i} \in \Dbc (Z_{c_i}^i \setminus Z_{c_i - c_n}^i)
\end{align*}
satisfy the axioms [AXS1] for perversity $\pp$.
\item there are isomorphisms of K\"unneth complexes
	\begin{align*}
 	(\iota_1)_* (\iota_1)^* \BB (X) &\xrightarrow{\cong} (\iota_1)_* \Q_{X \setminus X_{d-c_1}}, \\
	(\iota_1^i)_* (\iota_1^i)^* \BB (Z_{c_i}^i) &\xrightarrow{\cong} (\iota_1^i)_* \Q_{Z_{c_i}^i\setminus Z_{c_i-c_1}^i}
	\end{align*}
\item the adjunction morphisms $\BB(X) \to (\iota_n)_* (\iota_n)^* \BB(X)$ and $\BB(Z_{c_i}^i) \to (\iota_n^i)_* (\iota_n^i)^* \BB(Z_{c_i}^i)$ are isomorphisms of K\"unneth complexes.
\end{enumerate}
On the $Z_{c_i}^i$ we always use the compatible system of trivializations inherited by $X$ and also the induced K\"unneth structure on $\BB(Z_{c_i}^i)$ mentioned in Remark \ref{rem:inducedKunnethstr}.

Using a similar technique as in \cite[Theorem 7.3]{ABIS} and the definition of K\"unneth structures, we construct a sequence 
\[
\left( \KK_j(X), (\KK_j (Z_{c_1}^1, \beta_1, \ldots, Z_{c_k)^k}, \beta_k \right)_{j \in \left\{ 1, \ldots k+1 \right\}}
\]
of K\"unneth complexes such that $\KK_n$ satisfies the properties $[AXKS1_{\pp}]_n.$ 

$\mathbf{n=1.}$ We set $\KK_1 (X) := (\iota_1)_* \Q_{X \setminus X_{d-c_1}}$ and $\KK_1 (Z_{c_i}^i) := (\iota_1^i)_* \Q_{Z_{c_i}^i \setminus Z_{c_i - c_1}^i}$. We explain in Example~\ref{ex:Kunnethstr} that this gives rise to a K\"unneth complex and it is obvious that this complex satisfies the properties $[AXKS1_{\pp}]_1,$ since the spaces $X \setminus X_{d-c_1}$ and $Z_{c_i}^{i} \setminus Z_{c_i - c_1}^i$ are nonsingular.

\vspace{1ex}
$\mathbf{n \to n+1.}$ Let $\left( \KK_n(X), (\KK_n (Z_{c_1}^1), \beta_1, \ldots, \KK_n(Z_{c_k}^k), \beta_k \right)$ be a K\"unneth complex on $X$ that satisfies the properties $[AXKS1_{\pp}]_n$. We extend it to a K\"unneth complex that satisfies the properties $[AXKS1_{\pp}]_{n+1}$. We need to define $\KK_{n+1} (X)$ and $\KK_{n+1} (Z_{c_i}^i)$ for all $i \in \left\{ 1, \ldots, l \right\}.$
If $\mathbf{i<n}$ we define $\KK_{n+1} (Z^i_{c_i}):=\KK_{n} (Z^i_{c_i})$. 

For $\mathbf{i=n},$ the space $Z^n_{c_n}$ is a cone with vertex $v$. We denote by $j_v:\{v\}\to Z^n_{c_n}$ the closed inclusion and $i_{\rZ^n_{c_n}}:\rZ^n_{c_n}\to Z^n_{c_n}$ the open inclusion of the punctured cone. 

By \cite[Theorem 9.10]{ABIS}, the natural map 
$$\tau_{\leq \bar q(n)} {j_v}_* j_v^* {i_{\rZ^n_{c_n}}}_* \KK_{n} (Z^n_{c_n})|_{\rZ^n_{c_n}}\to {j_v}_* j_v^* {i_{\rZ^n_{c_n}}}_* \KK_{n} (Z^n_{c_n})|_{\rZ^n_{c_n}}$$
has a unique splitting 
$$\lambda: {j_v}_* j_v^* {i_{\rZ^n_{c_n}}}_*  \KK_{n} (Z^n_{c_n})|_{\rZ^n_{c_n}} \to \tau_{\leq \bar q(c_n)} {j_v}_* j_v^* {i_{\rZ^n_{c_n}}}_*\KK_{n} (Z^n_{c_n})|_{\rZ^n_{c_n}}.$$
Define $\KK_{n+1}(Z^n_{c_n}):=\cone{\lambda\circ\alpha}[-1]$, where
$$\alpha:{i_{\rZ^n_{c_n}}}_* \KK_{n} (Z^n_{c_n})|_{\rZ^n_{c_n}}\to {j_v}_* j_v^* {i_{\rZ^n_{c_n}}}_* \KK_{n} (Z^n_{c_n})|_{\rZ^n_{c_n}}$$
is the adjunction morphism.

For $\mathbf{i > n}$ we proceed as follows: by Property $(c_n)$, the adjunction $\KK_n (Z_{c_n}^n) \to (\iota_n^n)_* (\iota_n^n)^* \KK_n (Z_{c_n}^n)$ is an isomorphism. Then, the K\"unneth structure gives isomorphisms
\begin{align*} 
\KK_n (X)|_{TX_{d-c_n} \setminus X_{d-c_{n+1}}} & \cong \phi_n^* (\pi_1^n)^* \KK_n (Z_{c_n}^n) \cong \phi_n^* (\pi_1^n)^* (\iota_n^n)_* \KK_n (Z_{c_n}^n)|_{\rZ^n_{c_n}} \\
& = (\theta_n)_*(\pi^n_1\circ \phi_n|_{X_{d-c_n}\setminus X_{d-c_{n+1}}})^* j_v^* {i_{\rZ^n_{c_n}}}_*  \KK_{n} (Z^n_{c_n})|_{\rZ^n_{c_n}},
\end{align*}
where we used the equality $\pi_1^n \circ \phi_n \circ \theta_n = j_v \circ \phi_n|_{X_{d-c_n} \setminus X_{d-c_{n+1}}} \circ \pi_1^n$ in the second line. 
Define $\varphi_{\lambda}$ as the K\"unneth morphism given by composition of the adjunction morphism 
\[
a: \KK_n (X) \to (\theta_n)_* (\theta^n)^* \KK_n (X)
\]
 with the splitting 
\[
\begin{tikzcd}
(\theta_n)_*(\pi^n_1\circ \phi_n|_{X_{d-c_n}\setminus X_{d-c_{n+1}}})^* j_v^* {i_{\rZ^n_{c_n}}}_*  \KK_{n} (Z^n_{c_n})|_{\rZ^n_{c_n}} \ar{d} \\
(\theta_n)_*\tau_{\leq \bar q(c_n)}(\pi^n_1\circ \phi_n|_{X_{d-c_n}\setminus X_{d-c_{n+1}}})^* j_v^* {i_{\rZ^n_{c_n}}}_*  \KK_{n} (Z^n_{c_n})|_{\rZ^n_{c_n}}
\end{tikzcd}
\]
induced by $\lambda$. We then set 
\[ \KK_{n+1}(X) := (\iota_{n+1})_*(\iota_{n+1})^*\cone{\varphi_{\lambda}} [-1]. \]
Since the restriction $\theta_n: X_{d-c_n} \setminus X_{d-c_{n+1}} \hookrightarrow U_{k+1}$ is a closed inclusion and direct images of closed inclusions commute with truncations, the arguments in the proof of \cite[Theorem 7.3]{ABIS} (see p.39 therein) can be applied to deduce that $\KK_{n+1} (X)|_{X \setminus X_{d-c_{n+1}}}$ satisfies the [AXS1] properties. The same is true for the analogously defined $\KK_{n+1} (Z_{c_i}^i)|_{Z^i_{c_i - c_{n+1}}}$. 
Proposition~\ref{prop:basicpropertieskunneth} and Lemma~\ref{lem:cones} give, that 
\[ \left( \KK_{n+1}(X), (\KK_{n+1} (Z_{c_1}^1), \beta_1, \ldots, \KK_{n+1}(Z_{c_k}^k), \beta_k \right) \]
 is a K\"unneth complex.
Property~$(b)$ is satisfied, since Property~$(b)$ holds for $\KK_n(X)$ and there is an isomorphism of K\"unneth structures $\KK_{n}(X)\cong(\iota_{n})_*(\iota_{n})^*\KK_{n+1}(X)$. Property~$(c)$ for $\KK_{n+1}(X)$ holds because the composition $(\iota_{n+1})^*(\iota_{n+1})_*$ is equal to the identity.

Taking $n = k+1$, we proved the following.

\begin{mydef}
 Let $X_d\supset X_{d-c_1}\supset ....\supset X_{d-c_k}$ be a stratified pseudomanifold with a compatible system of tubular neighborhoods and a compatible system of trivializations $(\phi_{1},...,\phi_k)$. Let $\pp$ be a perversity. A K\"unneth intersection space complex 
 $(\IS (X),(\IS (Z^1_{c_1}),\beta_1,...,\IS (Z^k_{c_k}),\beta_k))$ for perversity $p$ (KIS-complex, for short) is a K\"unneth complex satisfying the following Properties $[AXKS1_{\pp}]$. 
 \begin{itemize}
 \item The complexes $\IS (X), ~ \IS (Z_{c_i}^i), ~ i \in \left\{ 1, \ldots, k \right\}$ satisfy the properties $[AXS1]$ for perversity $\pp$.
 \item there are K\"unneth isomorphisms 
	\begin{align*}
 	(\iota_1)_*(\iota_1)^*\IS (X) & \cong (\iota_1)_*\Q_{X\setminus X_{d-c_1}} ~ \text{and}\\
 	(\iota_1^i)_*(\iota_1^i)^*\IS (Z_{c_i}^i) & \cong (\iota_1^i)_*\Q_{Z_{c_i}^i \setminus Z_{c_i-c_1}^i}
	\end{align*}
\end{itemize}
(We again use the inherited compatible systems of trivializations and the induced K\"unneth structures on the $Z_{c_i}^i.$) 
\end{mydef}

\begin{prop}[Existence of KIS-complexes]\label{prop:KuennethIS_existence}
Given a stratified pseudomanifold with a compatible system of tubular neighborhoods and a compatible system of trivializations, a K\"unneth intersection space complex exists for any perversity.
\end{prop}

The next proposition shows that there is a unique (up to K\"unneth isomorphism) K\"unneth intersection space complex on $X$. 
\begin{prop}[Uniqueness of KIS-complexes]\label{prop:KuennethIS_uniqueness}\ \\
Let $\left( \LL(X), (\LL(Z_{c_1}^1, \beta_1, \ldots, \LL(Z_{c_k}^k), \beta_k \right)$ be a K\"unneth complex on $X$ that satisfies the properties $[AXKS1_{\pp}]$. Then, there is a K\"unneth isomorphism $\LL(X) \cong \IS (X).$ In other words, a KIS-complex for perversity $\pp$ is unique up to K\"unneth isomorphism.
\end{prop}
\begin{proof}
The proof is based on an induction showing that if $\LL_n(X)$ is a K\"unneth complex on $X$ satisfying the $[AXKS1_{\pp}]_n$ properties, then there is a K\"unneth isomorphism $\LL_n(X) \cong \KK_n (X)$ to the K\"unneth complex constructed at the beginning of Section \ref{subs:KIScomplex}.

\textbf{n = 1.} Axioms $(b_1)$ and $(c_1)$ imply, that $\LL_1(X) \cong (\iota_1)_* \Q_{X \setminus X_{d - c_1}} \cong \KK_1 (X).$

$\mathbf{n \to n+1.}$ Let $\left( \LL_{n+1}(X), \left( \LL_{n+1}(Z_{c_1}^1), \beta_1, \ldots, \LL_{n+1}(Z_{c_k}^k), \beta_k \right) \right)$ be a K\"unneth complex on $X$ that satisfies the axioms $[AXKS1_{\pp}]_{n+1}$. We consider the distinguished triangle associated to the following adjunction $f$.
\begin{equation}
\begin{tikzcd}
\LL_{n+1} (X)  \ar{r}{f} &  (\iota_n)_* (\iota_n)^* \LL_{n+1} (X) \ar{r} & cone(f) \ar{r}{[1]} & \
\end{tikzcd}
\label{triang:cone_adj}
\end{equation}
Since $(\iota_n)^* (\iota_n)_* = \id$, the restriction of the mapping cone of $f$ to $U_n$ vanishes, $(\iota_{n})^* cone(f) = 0$. Since $\iota_n$ is an open inclusion, one has $(\iota_n)^* = (\iota_n)^!$ and hence the adjunction triangle
\[ (\iota_n)_! (\iota_n)^* cone(f) \to cone(f) \to (j_n)_* (j_n)^* cone (f) \rightarrow{[1]} \]
implies that $cone (f) \cong (j_n)^* (j_n)^* cone (f).$ 
We apply the functor $(\iota_{n+1})_* (\iota_{n+1})^*$ to the above triangle (\ref{triang:cone_adj}) and use the equality of functors 
\[(\iota_{n+1})_* (\iota_{n+1})^* (j_n)_* (j_n)^* = (\theta_n)_* (\theta_n)^* \]
to arrive at the following triangle of K\"unneth structures.
\[
\begin{tikzcd}
\LL_{n+1} (X) \ar{r}{f} & (\iota_n)_* (\iota_n)^* \LL_{n+1} (X) \ar{r} & (\theta_n)_*(\theta_n)^* cone(f) \ar{r}{[1]} & \
\end{tikzcd}
\]
Since the K\"unneth complex $(\iota_n)_* (\iota_n)^* \LL_{n+1} (X)$ satisfies the properties $[AXKS1_{\pp}]_n$, by the induction assumption, there is a K\"unneth isomorphism
\[
\begin{tikzcd}
(\iota_n)_* (\iota_n)^* \LL_{n+1} (X) \ar{r}[swap]{\cong} \ar[bend left = 15]{rr}{\phi} & \KK_{n} (X) \ar[equal]{r} & (\iota_n)_* (\iota_n)^*  \KK_{n+1} (X),
\end{tikzcd}
\]
where the last equality follows from the definition of $\KK_{n+1} (X).$
 
Since $f$ is a K\"unneth morphism, it holds by Lemma \ref{lem:cones} that
\[ cone(f)|_{TX_{d - c_n} \setminus X_{d - c_{n+1}}} \cong \phi_n^* (\pi_1^n)^* cone(f_n), \]
with $f_n: \LL_{n+1} (Z_{c_n^n}) \to (\iota_n^n)_* (\iota_n^n)^* \LL_{n+1} (Z_{c_n}^n).$ Applying the adjunction functor $(\theta_n)_* (\theta_n)^*$ and using once more that $\pi_1^n \circ \phi_n \circ \theta_n| = j_v \circ \pi_1^n \circ \phi_n|,$ we get the isomorphism
\[ (\theta_n)_*(\theta_n)^* cone(f) \cong (\theta_n)_* \phi_n^* (\pi_1^n)^* j_v^* \, cone (f_n). \]
By Lemma \ref{lem:trivialidad}, this is isomorphic to the direct sum
\[ \bigoplus_{r=0}^d (\theta_n)_* (\underline{G_r})_{X_{d - c_n} \setminus X_{d - c_{n+1}}} [-r], \]
where $ G_r$  is the cohomology group $ H^r ( \tau_{\leq \qq(c_n)} (j_v)^* \KK_{n} (Z_{c_n}^n)).$
This fact follows from Property $(a_{n+1}),$ i.e. that $\LL_{n+1} (X)|_{X \setminus X_{d - c_{n+1}}}$ satisfies the [AXS1] properties: indeed, 
there is a direct sum decomposition 
\begin{align*}
(j_v)^* (\iota_n^n)_* (\iota_n^n)^* \LL_{n+1} (Z_{c_n}^n) \cong ~ & \tau_{\leq \qq (c_n)} (j_v)^* (\iota_n^n)_* (\iota_n^n)^* \LL_{n+1} (Z_{c_n}^n) \\
	& \oplus  \tau_{> \qq(c_n)} (j_v)^* (\iota_n^n)_* (\iota_n^n)^* \LL_{n+1} (Z_{c_n}^n),
\end{align*}
and the [AXS1] properties imply that the morphism 
\[ (j_v)^* \LL_{n+1} (Z_{c_n}^n) \to (j_v)^* (\iota_n^n)_* (\iota_n^n)^* \LL_{n+1} (Z_{c_n}^n) \]
has image isomorphic to the summand $\tau_{> \qq(c_n)} (j_v)^* (\iota_n^n)_* (\iota_n^n)^* \LL_{n+1} (Z_{c_n}^n)$. The existence of a K\"unneth isomorphism $(\iota_n^n)_* (\iota_n^n)^* \LL_{n+1} (Z_{c_n}^n) \to \KK_n (Z_{c_n}^n)$ implies the desired result. (This K\"unneth isomorphism is induced from the K\"unneth isomorphism $(\iota_n)_* (\iota_n)^* \LL_{n+1} (X) \xrightarrow{\cong} \KK_n (X).$)

Since the same arguments can be applied to the cone on the adjunction $g: \KK_{n+1} (X) \to (\iota_n)_* (\iota_n)^* \KK_{n+1} (X)$, we get a K\"unneth isomorphism
\[ 
\begin{tikzcd}
(\theta_n)_* (\theta_n)^* cone(f) \ar{r}{\cong} \ar{dr}[swap]{\psi} &  \bigoplus_{r=0}^d (\theta_n)_* (\underline{G_r})_{X_{d - c_n} \setminus X_{d - c_{n+1}}} [-r] \ar{d}{\cong} \\
 \ &  (\theta_n)_* (\theta_n)^* cone(g).
\end{tikzcd}
\]
We claim that the K\"unneth isomorphisms $\phi$ and $\psi$ fit into the following commutative diagram.
\begin{equation}
\begin{tikzcd}
 (\iota_n)_*(\iota_n)^*\LL_{n+1}(X)           \ar{r} \ar{d}{\phi} 	&  (\theta_{n})_*(\theta_{n})^*cone(f)  \ar{d}{\psi} \\
 (\iota_n)_*(\iota_n)^*\KK_{n+1}(X)           \ar{r}			&  (\theta_{n})_*(\theta_{n})^*cone(g)
\end{tikzcd}
\label{diag:comm_unique}
\end{equation}
The commutativity is obvious on all stalks over points that are not in $X_{d - c_n} \setminus X_{d- c_{n+1}}$. 
Thus, it is enough to check commutativity over $X_{d-c_{n}}\setminus X_{d-c_{n+1}}$. For $ x \in X_{d-c_{n}}\setminus X_{d-c_{n+1}}$, the following diagram commutes.
\[
\begin{tikzcd}
\left( 	(\iota_n)_* (\iota_n)^* \LL_{n+1} (X) \right)_x \ar{r} \ar{d}{\cong} \ar[bend right = 85]{ddd}{\phi_x}		& \quad (cone (f))_x  \quad \ar{d}{\cong} \ar[bend left = 85]{ddd}[swap]{\psi_x}	\\
\left( (\iota_n^n)_* (\iota_n^n)^* \LL_{n+1}(Z_{c_n}^n) \right)_v \ar{r}{proj.} \ar{d}{\cong} 				& 	\left( \tau_{\leq \qq(c_n)}  (\iota_n^n)_* (\iota_n^n)^* \LL_{n+1} (Z_{c_n}^n) \right)_v \ar{d}{\cong} \\
\left( (\iota_n^n)_* (\iota_n^n)^* \KK_{n+1} (Z_{c_n}^n) \right)_v \ar{r}{proj.} \ar{d}{\cong} 	& \left( \tau_{\leq \qq (c_n)} (\iota_n^n)_* (\iota_n^n)^* \KK_{n+1} (Z_{c_n}^n) \right)_v \ar{d}{\cong} \\
\left( 	(\iota_n)_* (\iota_n)^* \KK_{n+1} (X) \right)_x \ar{r}  							& 	\quad (cone (g))_x \quad
\end{tikzcd}
\]
Thus, Diagram (\ref{diag:comm_unique}) commutes.
We have constructed the following commutative diagram of K\"unneth structures
\[ 
\begin{tikzcd}
\KK_{n+1}(X)\ar{r}{f} & (\iota_n)_*(\iota_n)^*\KK_{n+1}(X)\ar{d}{\cong} \ar{r} & (\theta_{n})_*(\theta_{n})^*cone(f)\ar{r}{[1]}\ar{d}{\cong} &\quad   \\
\LL_{n+1}(X)\ar{r}{f} & (\iota_n)_*(\iota_n)^*\LL_{n+1}(X)              \ar{r} & (\theta_{n})_*(\theta_{n})^*cone(g)\ar{r}{[1]} &\quad
\end{tikzcd}
\]
where the rows are triangles and the vertical arrows isomorphisms. There is a K\"unneth isomorphism from $\KK_{n+1}(X)$ to $\LL_{n+1}(X)$ completing the triangle up to homotopy. Indeed the proof of Axioms I-III of triangulated categories for the category of bounded complexes with morphisms up to homotopy (see Theorem 4.1.9 of~\cite{GelfandManin}) adapts word by word to the case of K\"unneth structures and provides the desired completion.
\end{proof}

\begin{rem}
\label{rem:spacetosheaf}
An easy but tedious inspection on the construction of Section~3.3, Sections~4 and 5 up to Definition~5.15 and Theorem~5.16 of~\cite{ABIS} show that the complex that Definition 5.15 of~\cite{ABIS} associates with the intersection space Pair $(I^{\pp}X,I^{\pp}X_{d-2})$ of Remark~\ref{rem:compatIS} is the trivialized K\"unneth intersection space Complex of perversity $\pp$ for the compatible system of trivializations $(\phi_{1},...,\phi_k)$. It is remarkable that this happens for any choice of $(I^{\pp}X,I^{\pp}X_{d-2})$ as long as it is compatible with the system of trivializations $(\phi_{1},...,\phi_k)$.
\end{rem}
\proof
We limit ourself to highlight the main points of the tedious inspection of~\cite{ABIS} leading to the proof, the reaser interested in the details should have~\cite{ABIS} at hand. 

The construction of $(I^{\pp}X,I^{\pp}X_{d-2})$ in Section~3.3 of~\cite{ABIS} adds condition (iv), which implies that it has a compatible system of trivializations.

The pairs of spaces $(I^{\pp,n}X,I^{\pp,n}X_{d-2})$ defined in Section~4 are constructed in the same way as $(I^{\pp}X,I^{\pp}X_{d-2})$ but modifying smaller neighbourhoods of the strata. So, they also have a compatible system of trivializations.

For every $n\in \mathbb N$, $I^{\pp,n}X$ is included in a space $X'$ such that it has a compatible system of trivializations, it is topologically equivalent to $X$ and the retract $\pi:X' \to X$ preserves the system of trivializations (see  Definition~3.28 of~\cite{ABIS})

The complexes of sheaves $\mathcal{K}^{n, \bullet}$ defined in Section~5 are the kernel of the morphism
$$\nu^{n \#}: j^{n}_* \mathcal{C}_X^{n, \, \bullet} \rightarrow \mu^{n}_* \mathcal{C}_{X_{d-2}}^{n, \, \bullet}$$
where $\mathcal{C}_X^{n, \, \bullet}$ is the complex of of cubical singular cochains of $I^{\pp,n}X$, $\mathcal{C}_{X_{d-2}}^{n, \, \bullet}$ is the complex of of cubical singular cochains of $I^{\pp,n}X_{d-2}$ and
\begin{align*}
j^{n}: I^{\pp,n}X & ~\to X' \\
\mu^{n}: I^{\pp,n}X_{d-2} & ~\to X'
\end{align*}
are the canonical inclusions in $X'$.

It is easy to prove that these complexes of sheaves are K\"unneth complexes and the morphisms are K\"unneth morphisms.

Moreover, applying that the inverse limit commutes with the inverse image of sheaves, it is easy to prove that the inverse limit $\varprojlim_{n \in \mathbb N} \mathcal{K}^{n, \,\bullet}$ is also a K\"unneth complex.

Finally, since $\pi$ preserves the system of trivializations, we can check that $IS =\pi_* \varprojlim_{n \in \mathcal N} \mathcal{K}^{n, \,\bullet}$ is a K\"unneth complex.
\endproof
\begin{rem}\label{rem:deRhamIScomplex}
Another way to get the trivialized K\"unneth intersection space Complex of perversity $\pp$ for the compatible system of trivializations $(\phi_1, \ldots, \phi_k)$ on a smoothly stratified space is to use differential forms on the regular part $U = X \setminus X_{d-c_1}$. The complex of forms $\OI(X)$ was established in \cite{Ess} for a different class of pseudomanifolds of stratification depth two. The construction can be applied to our case as follows: a differential form $\omega$ on $U$ is contained in $ \OI (X)$ if 
\[ \omega|_{U \cap T_{d - c_i}} = \phi_i^{*} \sum_k \pi_1^* \eta_k \wedge \pi_2^* \gamma_k, \]
where the $\eta_k \in \Omega^\bullet (X_{d-c_i} \setminus X_{d - c_{i+1}})$ are forms on the singular stratum and the 
$\gamma_k \in \tau_{\geq m_i -1 -  \pp(m_i)} \OI (Z^i_{c_i} \setminus X_{d_{c_i}})$ are cotruncated $\OI$-forms on the link, with $m_i = \dim (Z^i_{c_i})$. Note that one has to clarify, why and how the (fiberwise) cotruncation used here can be iterated, but we do not elaborate on this. Following Banagl, see \cite[Section 6]{BandR}, the complex $\OI (X)$ is then used to define a pre-sheaf 
\[ V \mapsto \OI (V \cap U) = \left\{ \omega \in \Omega^\bullet (V \cap U) | \exists \widetilde{\omega} \in \OI (X): \widetilde{\omega}|_{V \cap U} = \omega \right\}.\] 
The sheafification $\mathbf{\OI}$ of this presheaf then is the desired de Rham description of the trivialized K\"unneth intersection space Complex of perversity $\pp$.
It might be interesting to compare our proof of Poincar\'e duality to the Poincar\'e duality of the de Rham picture, which is induced by integrating wedge products of forms over $U$, see \cite[Theorem 8.2]{BandR} and \cite[Theorem 7.4.1]{Ess}.
\end{rem}
It will be convenient to have the ``dual'' of the previous Proposition:
\begin{prop}\label{prop:KuennethISdual}
Let $X_d\supset X_{d-c_1}\supset ....\supset X_{d-c_k}$ be a stratified pseudomanifold with a compatible system of tubular neighborhoods and a compatible system of trivializations $(\phi_{1},...,\phi_k)$. There is a (up to K\"unneth isomorphism) unique K\"unneth complex $('\IS (X),('\IS (Z^1_{c_1}),\beta'_1,...,'\IS (Z^k_{c_k}),\beta'_k))$ satisfying the following properties $[AXKS1'_{\pp}]$. 
\begin{itemize}
 \item The complexes $'\IS (X)$ and $'\IS (Z^i_{c_i}), ~ i \in \left\{ 1, \ldots, k \right\}$ satisfy the properties $[AXS1]$ for perversity $\pp$ of \cite[Section 6]{ABIS}.
 \item there are K\"unneth isomorphisms 
	\begin{align*}
	(\iota_1)_!(\iota_1)^! ('\IS (X)) & \cong (\iota_1)_!\Q_{X\setminus X_{d-c_1}}, \\
	(\iota_1^i)_!(\iota_1^i)^! ('\IS (Z_{c_i}^i)) & \cong (\iota_1^i)_!\Q_{Z_{c_i}^i\setminus Z_{c_i-c_1}^i}, ~ i \in \left\{ 1, \ldots, k \right\}.
	\end{align*}
\end{itemize}
\end{prop}
\proof
Let $\bar q$ be the complementary perversity to $\bar p$. Define 
$$'\IS(X):=\VD\ISq(X)[-d],\quad\quad'\IS(Z^i_{c_i}):=\VD\ISq(X)(Z^i_{c_i})[-c_i]$$
for any $i\in\{1,...,k\}$. By Theorem~10.1 of~\cite{ABIS} the complexes defined above satisfy the axioms $[AXS1]$ for perversity $\bar p$. 
We claim that $'\IS(X)$ has a K\"unneth structure as predicted in the proposition for $\beta'_i:=(\VD(\beta_i)[-d])^{-1}$. Indeed, since $\ISq(X)$ has a K\"unneth structure $(\ISq (Z^1_{c_1}),\beta_1,...,\ISq (Z^k_{c_k}),\beta_k)$, for any $i$ we have an isomorphism 
$$\beta_i: \ISq(X)|_{TX_{d-c_i}\setminus X_{d-c_{i+1}}}\xrightarrow{\cong} \phi_i^* (\pi^i_1)^* \ISq(Z^i_{c_i}).$$
Applying the functor $\VD(\centerdot)[-d]$ we obtain an isomorphism 
$$(\VD(\beta_i)[-d])^{-1}: ~ '\IS(X)|_{TX_{d-c_i}\setminus X_{d-c_{i+1}}}\xrightarrow{\cong} \phi_i^! (\pi^i_1)^!\VD(\ISq(Z^i_{c_i}))[-d].$$
Since $\phi_i$ is an isomorphism we have $(\phi_i)^!=(\phi_i)^*$. Moreover, since $\pi^k_1: X_{d-c_k} \times Z^k_{c_k} \to Z$ is a projection with fiber the smooth $(d-c_k)$-dimensional manifold $X_{d-c_k}$ we have $(\pi^k_1)^! = (\pi^k_1)^* [d-c_k]$ (see e.g. \cite[p. 193]{Ban07}). 
Then 
$$(\pi^i_1)^!\VD(\ISq(Z^i_{c_i}))[-d]=(\pi^i_1)^*\VD(\ISq(Z^i_{c_i})[-c_k]='\IS(Z^i_{c_i}).$$
This proves the claim.

By Proposition~\ref{prop:KuennethIS_uniqueness}, there is an isomorphism of K\"unneth structures 
$$(\iota_1)_*\Q_{X\setminus X_{d-c_1}}\to (\iota_1)_*(\iota_1)^* (\IS (X)).$$
Applying $\VD(\centerdot)[-d]$, we get an isomorphism in the derived category
$$(\iota_1)_!\Q_{X\setminus X_{d-c_1}}\to (\iota_1)_!(\iota_1)^! ('\IS (X)).$$
It is a routine check as above that it lifts to an isomorphism of K\"unneth structures. This proves existence.

Uniqueness follows dualizing, twisting with $[-d]$, and applying uniqueness of the K\"unneth intersection space complex. 
\endproof

A corollary of the above proof is the following duality

\begin{cor}
There is a K\"unneth isomorphism 
$$'\ISq (X)\cong \VD \IS (X)[-d].$$
\end{cor}


\begin{prop}
There is a K\"unneth isomorphism
$$\IS(X) ~\cong ~'\IS(X).$$
\end{prop}
\proof
Applying the functor $(\iota_1)_!(\iota_1)^!$ to the isomorphism of K\"unneth structures 
$(\iota_1)_*\Q_{X\setminus X_{d-c_1}}\to (\iota_1)_*(\iota_1)^* (\IS (X))$
and using that $(\iota_1)^*=(\iota_1)^!$ and that $(\iota_1)^*(\iota_1)_*$ is equal to the identity we obtain the isomorphism of K\"unneth structures 
$$(\iota_1)_!\Q_{X\setminus X_{d-c_1}}\to (\iota_1)_!(\iota_1)^! (\IS (X)).$$
Taking the inverse we show that $\IS (X)$ satisfies the conditions of Proposition~\ref{prop:KuennethISdual}.
\endproof

Combining the results above our main theorem follows immediately.

\begin{thm}
\label{th:main1}
Let $X_d\supset X_{d-c_1}\supset ....\supset X_{d-c_k}$ be a stratified pseudomanifold with a compatible system of tubular neighborhoods and a compatible system of trivializations $(\phi_{1},...,\phi_k)$. Let $\pp, \qq$ be complementary perversities and let $\IS (X)$ be the K\"unneth intersection space complex of perversity $\pp$ for the given system of trivializations. Then, $\VD \IS [-d],$ where $\VD$ denotes the Verdier dual, is the K\"unneth intersection space complex of perversity $\qq$. 

In particular, if $X_d\supset X_{d-c_1}\supset ....\supset X_{d-c_k}$ is a Witt space, for the middle perversity we obtain a unique self-dual K\"unneth intersection space complex.
\end{thm}

The proof of the Main Theorem is an immediate consequence of the results above and Remarks~\ref{rem:compatIS} and~\ref{rem:spacetosheaf}.  

\begin{rem}
\label{rem:mhm}
If $X_d\supset X_{d-c_1}\supset ....\supset X_{d-c_k}$ is a filtration by algebraic varieties forming a stratified pseudomanifold with a compatible system of tubular neighborhoods and a compatible system of trivializations $(\phi_{1},...,\phi_k)$, then all the constructions of the paper can be performed in the derived category of mixed Hodge modules. So the K\"unneth intersection space complex for a fixed perversity is in the derived category of mixed Hodge modules, and is unique up to isomorphism in this category.
\end{rem}

\end{document}